\documentclass[12pt]{amsart}
\usepackage[colorlinks=true,citecolor=blue,linkcolor=blue]{hyperref}
\usepackage{amssymb,amscd,amsmath,graphicx,mathtools}
\usepackage{enumerate}
\usepackage{fullpage}
\usepackage{comment}
\usepackage{xcolor}
\usepackage[all]{xy}
\usepackage{cleveref}
\usepackage{tikz-cd}
\usepackage{enumerate}
\usepackage{fullpage}
\usepackage{xcolor}
\usepackage{amsthm}
\usepackage{float}
\usepackage{pgfplots}
\usepackage{listings}
\usepackage{booktabs}
\usepackage{comment}
\usepackage{graphicx}
\newtheorem{alphthm}{Theorem}

\usepackage{csquotes}

\numberwithin{equation}{section}
\newtheorem{thm}{Theorem}[section]

\newtheorem{cor}[thm]{Corollary}
\newtheorem{prop}[thm]{Proposition}
\newtheorem{lem}[thm]{Lemma}

\newtheorem{quest}[thm]{Question}

\theoremstyle{definition}
\newtheorem{defn}[thm]{Definition}
\newtheorem{rem}[thm]{Remark}
\newtheorem{exmp}[thm]{Example}

\newtheorem{setup}[thm]{Setup}

\newcommand{\mfrak}{\mathfrak m}
\newcommand{\bI}{\overline{I}}
\newcommand{\im}{\operatorname{im}}
\newcommand{\rank}{\operatorname{rank}}
\newcommand{\HS}{\operatorname{HS}}
\newcommand{\len}{\operatorname{length}}

\DeclareMathOperator{\height}{ht}

\DeclareMathOperator{\supp}{supp}

\DeclareMathOperator{\conv}{convexhull}

\newcommand{\CC}{{\mathbb C}}

\newcommand{\ZZ}{{\mathbb Z}}
\newcommand{\NN}{{\mathbb N}}

\newcommand{\RR}{{\mathbb R}}

\newcommand{\FF}{{\mathbb F}}

\def\a{{\bf a}}

\def\x{{\bf x}}

\def\1{{\bf 1}}
\def\0{{\bf 0}}

\begin{document}
\title[Weakly Newton-nondegenerate binomial ideals]{Weakly Newton-nondegenerate binomial ideals}

\author{Takayuki Hibi}
\address{Department of Pure and Applied Mathematics, Graduate School of Information
Science and Technology, The University of Osaka, Suita, Osaka 565-0871, Japan}
\email{hibi@math.sci.osaka-u.ac.jp}

\author{Vinh Anh Ph{\d{A}}m}
\address{Tulane University, Department of Mathematics, 6823 St. Charles Ave., New Orleans, LA 70118, USA}
\email{vpham1@tulane.edu}



\begin{abstract}
Following \cite{HNP}, an ideal $I$ of a polynomial ring $R=k[x_1,\dots,x_n]$ is called \emph{weakly
Newton-nondegenerate}, or weakly NND, if its integral closure $\overline{I}$ is a monomial ideal.
We study weak Newton nondegeneracy for the family of quadratic binomial ideals
\[
I=(x_1^2+\epsilon_1x_{a_1}x_{b_1},\ \dots,\ x_n^2+\epsilon_nx_{a_n}x_{b_n}),\qquad \epsilon_i\in\{\pm1\},\ a_i\neq b_i,
\]
over an algebraically closed field. We prove that $I$ is weakly NND if and only if $\overline I=\mfrak^2$, if and only if the given generators form a regular sequence, and if and only if an explicit combinatorial condition on the pair (support pattern, sign pattern) holds: no nonempty subset $S\subseteq\{1,\dots,n\}$ is simultaneously \emph{closed} for the support data and \emph{sign-trivial} for the associated lattice of relations. The last equivalence rests on a solvability criterion for systems of monomial equations over a divisible abelian group, in the spirit of Eisenbud and Sturmfels. As an application we classify all such ideals for $n=3$.
\end{abstract}

\maketitle

\section{Introduction} \label{sec.intro}

Newton polyhedra provide a bridge between the algebra of an ideal and the combinatorics of its exponents. The bridge is at its strongest for \emph{nondegenerate} objects, where invariants that would otherwise be analytic or homological become computable from a polytope. This principle goes back to Kouchnirenko's computation of the Milnor number of a nondegenerate singularity in terms of the Newton number \cite{Kou76}, and it has been developed for ideals by Bivi\`a-Ausina, Fukui and Saia \cite{BAFS02}, among others.

Throughout, $R=k[x_1,\dots,x_n]$ is a polynomial ring over a field $k$ and $\mfrak=(x_1,\dots,x_n)$.

\begin{defn}[{\cite[Remark 4.15]{HNP}}]\label{def:NND}
An ideal $I\subseteq R$ is \emph{weakly Newton-nondegenerate}, abbreviated \emph{weakly NND}, if its integral closure $\bI$ is a monomial ideal.
\end{defn}

The qualifier \enquote{weakly} marks a genuine distinction. In a regular local ring, or in a ring of formal power series or of convergent power series, an ideal is Newton-nondegenerate if and only if its integral closure is monomial \cite[Theorem A]{HNP}, \cite[Theorem 2.3]{BAFS02}, \cite{Sai96}. In a polynomial ring the two conditions genuinely differ, and \cite[Remark 4.15]{HNP} reserves \enquote{weakly NND} for the integral-closure condition, which is the one we study here.

The condition in Definition \ref{def:NND} is a strong one: it says that the passage to the integral closure destroys all the non-monomial information in $I$. Consequently $\bI$, and with it every invariant of $I$ that only depends on $\bI$, such as the multiplicity, the analytic spread, or the asymptotic Samuel function, is computed by the Newton polyhedron $\Gamma_+(I)$. It is therefore natural to ask, for a concrete family of non-monomial ideals, exactly which members are weakly NND. The purpose of this paper is to answer that question for the following family, which is small enough to be classified completely and rich enough that the answer is not obvious.

\begin{setup}\label{setup}
Let $k$ be an algebraically closed field, $R=k[x_1,\dots,x_n]$ with $n\geq2$, and let
\[
I=(g_1,\dots,g_n),\qquad g_i=x_i^2+\epsilon_i x_{a_i}x_{b_i},
\]
where $\epsilon_i\in\{\pm1\}$ and $a_i,b_i\in\{1,\dots,n\}$ with $a_i\neq b_i$, for each $i=1,\dots,n$. We call $(m_1,\dots,m_n)$, $m_i:=x_{a_i}x_{b_i}$, the \emph{support pattern} and $(\epsilon_1,\dots,\epsilon_n)$ the \emph{sign pattern} of $I$. We allow $i\in\{a_i,b_i\}$, so that $x_i^2\pm x_ix_j$ is a permitted generator.
\end{setup}

Computation with \texttt{Macaulay2} \cite{M2} for $n\leq3$ suggests some patterns immediately. For
example, if all of the signs are negative, then $I$ is never weakly NND; this is correct, and it holds for every $n$ (Corollary \ref{cor:allminus}). Our main theorem, which captures completely the combinatorial condition governing weak Newton nondegeneracy, is the following.

\begin{alphthm}\label{thm:main-intro}
In the situation of Setup \ref{setup}, let $e_1,\ldots,e_n$ be the standard basis vectors of $\ZZ^n$ and write $w_i=2e_i-e_{a_i}-e_{b_i}\in\ZZ^n$. 
The following are equivalent:

\begin{enumerate}[(1)]
\item $I$ is weakly NND;
\item $\bI=\mfrak^2$;
\item $g_1,\dots,g_n$ is a regular sequence;
\item there is no nonempty $S\subseteq\{1,\dots,n\}$ which is \emph{closed}, i.e.
\[
i\in S\Rightarrow a_i,b_i\in S,\qquad i\notin S\Rightarrow\{a_i,b_i\}\not\subseteq S,
\]
and \emph{sign-trivial}, i.e. $\prod_{i\in S}(-\epsilon_i)^{\lambda_i}=1$ for every $\lambda = (\lambda_i)_{i \in S}$ in the relation lattice 
\[
L_S=\{\lambda=(\lambda_i)_{i \in S}\in\ZZ^S:\sum_{i\in S}\lambda_iw_i=0\}.
\]
\end{enumerate}
\end{alphthm}

Theorem \ref{thm:main-intro} is proved in Section \ref{sec:torus}. The equivalence of (1), (2) and (3) follows from Rees's multiplicity theorem, once one observes that $\mfrak^2$ is the only integrally closed monomial ideal that $\bI$ could possibly be; this uses the material assembled in Sections \ref{sec:prelim} and \ref{sec:CI}. The substance of the theorem lies in the equivalence of (3) and (4), which converts a question about the vanishing locus $V(I)$ into a question about characters on a lattice: we stratify $V(I)$ according to the support of a point and appeal, on each stratum, to a solvability criterion for systems $p^{w_i}=c_i$ over a divisible abelian group (Lemma \ref{lem:divisible}), which puts an observation of Eisenbud and Sturmfels \cite[\S2]{ES96} into the biconditional form we require; it is applied to the individual torus strata of $V(I)$ in Lemma \ref{lem:stratify}.

Condition (4) is a finite, purely combinatorial test on the pair (support pattern, sign pattern). Section \ref{sec:n3} carries it out completely for $n=3$, where the support patterns fall into exactly three cases; the resulting classification (Corollary \ref{cor:n3}) is:
\begin{itemize}
\item if $m_1,m_2,m_3$ are pairwise distinct, then $I$ is weakly NND if and only if $\epsilon_1\epsilon_2\epsilon_3=+1$;
\item if some monomial repeats, then $I$ is weakly NND if and only if $\epsilon_i\epsilon_j=-1$, where $\{i,j\}$ is an explicitly determined \emph{critical pair} of positions, the remaining sign being free.
\end{itemize}
Section \ref{sec:n4} shows, with two ideals for $n=4$, that the quantifier over $S$ in condition (4) is genuinely needed, and closes with a question on its complexity and on generators of higher degree.

\section{Preliminaries}\label{sec:prelim}

Recall that $r\in R$ is \emph{integral} over an ideal $I$ if it satisfies an equation
\begin{equation}\label{eq:integraldep}
r^d+c_1r^{d-1}+\dots+c_d=0,\qquad c_j\in I^{j},
\end{equation}
and that the set $\bI$ of such elements is an ideal, the \emph{integral closure} of $I$, which is itself integrally closed \cite[Chapter 1]{SH06}.

Following \cite[Definition 2.8]{HNP}, for an ideal $I\subseteq R$ we write $$\supp(I)=\bigcup_{f\in I}\supp(f)\subseteq\NN^n$$ for the union of the exponent supports of the elements of $I$, and
\[
\Gamma_+(I)=\conv\bigl\{u+v:\ u\in\supp(I),\ v\in\RR^n_{\geq0}\bigr\}\subseteq\RR^n_{\geq0}
\]
for the \emph{Newton polyhedron} of $I$. We write $\Gamma_x$ for the same construction in a general regular local ring $A$ with regular system of parameters $x$; when $A=R$ and $x=x_1,\dots,x_n$ this is exactly $\Gamma_+$ above.

Let $(A,\mathfrak n)$ be a regular local ring with regular system of parameters $x=x_1,\dots,x_d$. For a closed convex $\Delta\subseteq\RR^d_{\geq0}$ let $C(\Delta)$ be the cone over $\Delta$ with vertex the origin, and let $$A_\Delta=\{g\in A:\supp(g)\subseteq C(\Delta)\cap\ZZ^d_{\geq0}\}$$ with maximal ideal $\mathfrak n_\Delta$. With $g = \sum\limits_{\a \in \supp(g)} c_\a\x^\a\in A,$
set $g_\Delta = \sum_{\a \in \supp(g) \cap C(\Delta)} c_\a \x^\a.$ An ideal $J\subseteq A$ is \emph{Newton-nondegenerate} (NND) if it has generators $g_1,\dots,g_s$ such that $J_\Delta=(g_{1\Delta},\dots,g_{s\Delta})$ is $\mathfrak n_\Delta$-primary in $A_\Delta$ for every compact face $\Delta$ of $\Gamma_x(J)$ \cite[Definition 3.1]{HNP}. Over $\CC[[x_1,\ldots,x_d]]$ this says exactly that for each compact face $\Delta_1$ of each compact face $\Delta$ the truncated system $g_{1\Delta_1}=\dots=g_{s\Delta_1}=0$ has no solution in $(\CC^\times)^d$ \cite[Remark 3.2]{HNP}, \cite[Theorem 6.2]{Kou76} and \cite{Sai96,BAFS02}. By \cite[Theorem A]{HNP}, in a regular local ring
\[
J\ \text{is NND}\iff\overline J\ \text{is a monomial ideal in }x.
\]

In a polynomial ring, the two properties are no longer equivalent, which is the reason for the terminology of Definition \ref{def:NND}: for example, by \cite[Remark 3.5]{HNP}, the ideal $I=(x^5+xy^3,\,y^5+x^3y)\subseteq\CC[x,y]$ satisfies the face condition and $\overline{I\CC[[x,y]]}$ is monomial, while $\bI$ is not a monomial ideal in $\CC[x,y]$. Only one implication survives.

\begin{prop}\label{prop:completion}
Let $I\subseteq R=k[x_1,\dots,x_n]$ be an ideal. If $I$ is weakly NND in $R$, then $I\hat R$ is NND in $\hat R=k[[x_1,\dots,x_n]]$, and $IR_\mfrak$ is NND in $R_\mfrak$.
\end{prop}

\begin{proof}
The $\mfrak$-adic completion map $R\to\hat R$ is flat. Since $R$ is excellent, its formal fibers are geometrically regular, so this map is a normal homomorphism in the sense of \cite[Definition 19.0.1]{SH06}; by \cite[Corollary 19.5.2]{SH06}, $\bI\hat R=\overline{I\hat R}$. If $\bI$ is generated by monomials, then so is $\overline{I\hat R}$, and $I\hat R$ is NND by \cite[Theorem A]{HNP}. The same argument applies to $R\to R_\mfrak$, using \cite[Proposition 1.1.4]{SH06}.
\end{proof}

For an $\mfrak$-primary ideal $J$ of a Noetherian local ring $(A,\mfrak)$ of dimension $d$, $e(J)$ denotes the Hilbert--Samuel multiplicity. We use two standard facts \cite[Ch.~11]{SH06} and \cite[\S4.6--4.7]{BH98}: if $A$ is Cohen--Macaulay and $J$ is generated by a system of parameters, then $e(J)=\len(A/J)$; and $e(\mfrak^t)=t^d e(\mfrak)$. We use Rees's theorem as a key input in determining the integral closure of the ideals of Setup \ref{setup}.

\begin{thm}[Rees \cite{Rees61}; see \textup{\cite[\S11.3]{SH06}}]\label{thm:rees}
Let $(A,\mfrak)$ be a formally equidimensional Noetherian local ring and $J\subseteq I$ ideals with $J$ $\mfrak$-primary. Then $\overline J=\overline I$ if and only if $e(J)=e(I)$.
\end{thm}

We fix the notation of Setup \ref{setup} for the rest of the paper. Let $e_1,\ldots,e_n$ be the standard basis vectors of $\ZZ^n$. For $i=1,\dots,n$ put
\[
w_i:=2e_i-e_{a_i}-e_{b_i}\in\ZZ^n .
\]
 For $p\in(k^\times)^n$ and $w\in\ZZ^n$ we write $p^w=\prod_kp_k^{w_k}$, so that $p^{w_i}=p_i^2/(p_{a_i}p_{b_i})$.

For a subset $S\subseteq\{1,\dots,n\}$ satisfying the closure condition of Theorem \ref{thm:main-intro}(4) we regard $w_i\in\ZZ^S$ for $i\in S$, which is legitimate since then $a_i,b_i\in S$, and we let
\[
M_S\colon\ZZ^S\to\ZZ^S,\qquad M_S(e_i)=w_i,
\]
be the homomorphism whose \emph{columns} are the $w_i$, we set $$L_S:=\ker M_S=\{\lambda=(\lambda_i)_{i \in S} \in\ZZ^S:\sum_{i\in S}\lambda_iw_i=0\},$$ and we define the \emph{sign character}
\[
\sigma_S\colon L_S\to\{\pm1\},\qquad \sigma_S(\lambda)=\prod_{i\in S}(-\epsilon_i)^{\lambda_i},
\]
which is a group homomorphism. The second condition of Theorem \ref{thm:main-intro}(4) reads $\sigma_S\equiv1$. We call $S$ \emph{closed} if the first condition holds and \emph{sign-trivial} if $\sigma_S\equiv1$, and we call $S$ a \emph{witness} if it is nonempty, closed and sign-trivial. Thus Theorem \ref{thm:main-intro}(4) says: \emph{$I$ has no witness}.

\section{Regular sequences and the zero locus}\label{sec:CI}

This section provides the geometric reformulation of the regular sequence condition that will be used in the proof of Theorem \ref{thm:main-intro}.

\begin{lem}\label{lem:regseq}
Let $g_1,\dots,g_n\in\mfrak$ be homogeneous of positive degree. Then $g_1,\dots,g_n$ is a regular sequence in $R$ if and only if $V(g_1,\dots,g_n)=\{0\}$ in $k^n$.
\end{lem}

\begin{proof}
 ($\Rightarrow$) Let $g_1,\dots,g_n\in\mfrak$ be a regular sequence of homogeneous polynomials in $R$. Since $R$ is Cohen--Macaulay, $\dim R/(g_1,\ldots,g_n)=\dim R - n=0$, so $V(g_1,\dots,g_n)$ is finite and nonempty. Being the zero locus of homogeneous polynomials, it is a cone; hence, it equals $\{0\}$.

($\Leftarrow$) Let $J=(g_1,\ldots,g_n)$. By the Nullstellensatz, $V(J)=\{0\}$ gives $\sqrt J=\mfrak$, so $\height (J)=n$. In a Cohen--Macaulay local ring, elements generating an ideal of height equal to their number automatically form a regular sequence (see \cite[Theorem 2.1.2]{BH98}). Since $R_\mfrak$ is Cohen--Macaulay, $g_1,\dots,g_n$ is a regular sequence in $R_\mfrak$; since $g_i$ are homogeneous, this remains a regular sequence in $R$.
\end{proof}

\begin{exmp}\label{ex:cyclic}
Assume $\operatorname{char}k\ne2$. For $I=(x^2+yz,\,y^2+xz,\,z^2+xy)\subseteq k[x,y,z]$ one checks (e.g. by Corollary \ref{cor:n3} or directly) that the only common zero is the origin, i.e.\ $V(I)=\{0\}$. By Theorem \ref{thm:main-intro}, this gives $\bI=\mfrak^2=(x^2,xy,xz,y^2,yz,z^2)$, and indeed $e(IR_\mfrak)=2^3=8=e(\mfrak^2R_\mfrak)$.
\end{exmp}

\section{The torus condition}\label{sec:torus}

This section supplies the last ingredients for the proof of Theorem~\ref{thm:main-intro}: we analyze the zero locus $V(I)$ stratum by stratum according to the support of a point, and translate the existence of a nonzero point into the combinatorial condition (4). The two ingredients are Lemmas \ref{lem:divisible} and \ref{lem:stratify}.

\subsection{A solvability criterion over divisible groups and stratifying the zero locus}

Systems of monomial equations over a field behave like linear systems over $\ZZ$, except that the group $k^\times$ is not free; what replaces freeness is divisibility, i.e., injectivity as a $\ZZ$-module. The following lemma is the precise statement we need. It generalizes the observation, used by Eisenbud and Sturmfels in their study of binomial ideals \cite[\S2]{ES96}, that a system of monomial equations over an algebraically closed field is solvable as soon as it is consistent.

\begin{lem}\label{lem:divisible}
Let $A$ be a divisible abelian group, written multiplicatively, and let  
$w_1,\dots,w_N$ belong to $\ZZ^m$ with $N,m\geq1$. Write $M\colon\ZZ^N\to\ZZ^m$ for the
homomorphism determined by $M(e_i)=w_i$, and set $L=\ker M$. Then, for
$c=(c_1,\dots,c_N)\in A^N$, the system
\[
p^{w_i}=c_i\qquad(i=1,\dots,N)
\]
has a solution $p\in A^m$ if and only if $\prod_{i=1}^Nc_i^{\lambda_i}=1$ for every $\lambda\in L$.
\end{lem}

\begin{proof}
We use the standard isomorphism
\[
A^m\xrightarrow{\ \sim\ }\operatorname{Hom}(\ZZ^m,A),\qquad p\mapsto\Bigl(f_p:\lambda\mapsto\textstyle\prod_kp_k^{\lambda_k}\Bigr),
\]
and similarly $A^N\cong\operatorname{Hom}(\ZZ^N,A)$. Let $\Phi\colon A^m\to A^N$, $\Phi(p)=(p^{w_1},\dots,p^{w_N})$. Since $f_p(w_i)=p^{w_i}$ and $M(e_i)=w_i$, under the two isomorphisms $\Phi$ becomes
\[
\operatorname{Hom}(\ZZ^m,A)\to\operatorname{Hom}(\ZZ^N,A),\qquad f\mapsto f\circ M .
\]
Thus the system is solvable exactly when $f_c\in\im(-\circ M)$.

($\Rightarrow$) If $f_c=f\circ M$ then for $\lambda\in L$ we get $\prod_ic_i^{\lambda_i}=f_c(\lambda)=f(M\lambda)=f(0)=1$.

($\Leftarrow$) Suppose $f_c|_L$ is trivial. Then $f_c$ factors as $f_c=\hat g\circ\pi$ for a unique $\hat g\colon\ZZ^N/L\to A$, where $\pi\colon\ZZ^N\twoheadrightarrow\ZZ^N/L$ is the quotient map. The first isomorphism theorem provides an isomorphism $\theta\colon\ZZ^N/L\xrightarrow{\sim}\im(M)$ with $\iota\circ\theta\circ\pi=M$, where $\iota\colon\im(M)\hookrightarrow\ZZ^m$ is the inclusion. Put $g:=\hat g\circ\theta^{-1}\colon\im(M)\to A$. Since $A$ is divisible it is an injective $\ZZ$-module, so $g$ extends along $\iota$ to some $f\colon\ZZ^m\to A$ with $f\circ\iota=g$. Then
\[
f\circ M=f\circ\iota\circ\theta\circ\pi=g\circ\theta\circ\pi=\hat g\circ\pi=f_c ,
\]
so $c=\Phi(p)$ for the tuple $p=(f(e_1),\dots,f(e_m))$.
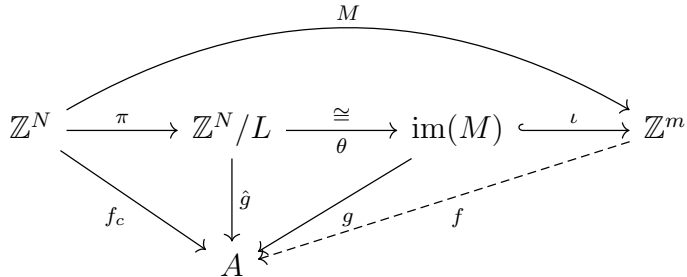
\[
\begin{tikzcd}[column sep=large, row sep=large]
\ZZ^N
  \arrow[r, "\pi"]
  \arrow[rrr, bend left=30, "M"]
  \arrow[dr, swap, "f_c"]
& \ZZ^N/L
  \arrow[r, "\theta"', "\cong"]
  \arrow[d, "\hat g"]
& \im(M)
  \arrow[r, hook, "\iota"]
  \arrow[dl, "g"]
& \ZZ^m
  \arrow[dll, dashed, "f"]
\\
& A & &
\end{tikzcd}
\qedhere
\]
\end{proof}

Since $k$ is algebraically closed, $k^\times$ is divisible, and Lemma \ref{lem:divisible} applies with $A=k^\times$ in any characteristic.

\begin{lem}\label{lem:stratify}
Let $p\in k^n\setminus\{0\}$ and let $S=\{i:p_i\neq0\}$ be its support. Then $p\in V(I)$ if and only if $S$ is closed and $q:=p|_S\in(k^\times)^S$ satisfies
\[
q^{w_i}=-\epsilon_i\qquad\text{for all }i\in S .
\]
\end{lem}

\begin{proof}
($\Rightarrow$) Let $i\in S$. From $g_i(p)=0$ we get $p_{a_i}p_{b_i}=-\epsilon_i^{-1}p_i^2=-\epsilon_ip_i^2\neq0$, so $a_i,b_i\in S$ and, dividing, $q^{w_i}=p_i^2/(p_{a_i}p_{b_i})=-\epsilon_i$. Let $i\notin S$. Then $p_i=0$, so $0=g_i(p)=\epsilon_ip_{a_i}p_{b_i}$ forces $p_{a_i}p_{b_i}=0$, i.e. $\{a_i,b_i\}\not\subseteq S$.

($\Leftarrow$) Let $i\in S$. Closedness gives $a_i,b_i\in S$, and $q^{w_i}=-\epsilon_i$ gives $p_{a_i}p_{b_i}=-\epsilon_ip_i^2$, so
\[
g_i(p)=p_i^2+\epsilon_i(-\epsilon_ip_i^2)=(1-\epsilon_i^2)p_i^2=0 .
\]
Let $i\notin S$. Then at least one of $a_i,b_i$ lies outside $S$, so $p_{a_i}p_{b_i}=0$ and $g_i(p)=p_i^2+\epsilon_ip_{a_i}p_{b_i}=0$. Hence $p\in V(I)$.
\end{proof}

\begin{proof}[Proof of Theorem~\ref{thm:main-intro}]
\emph{(1)$\Rightarrow$(2).} Assume $\bI$ is a monomial ideal. Since $g_i\in\bI$ and $\bI$ is monomial, $x_i^2\in\bI$ for every $i$. For $i\ne j$ the element $r=x_ix_j$ satisfies the equation of integral dependence
\[
r^2+0\cdot r-x_i^2x_j^2=0,\qquad x_i^2x_j^2\in\bI^{\,2},
\]
so $r\in\overline{\bI}=\bI$. Hence $\mfrak^2\subseteq\bI$. Conversely $I\subseteq\mfrak^2$ and $\mfrak^2$ is integrally closed, so $\bI\subseteq\mfrak^2$. Therefore $\bI=\mfrak^2$.

\emph{(2)$\Rightarrow$(3).} If $\bI=\mfrak^2$ then $V(I)=V(\bI)=V(\mfrak^2)=\{0\}$, and Lemma \ref{lem:regseq} applies.

\emph{(3)$\Rightarrow$(2).} Suppose $g_1,\dots,g_n$ is a regular sequence. The Koszul complex on $g_1,\dots,g_n$ then resolves $R/I$, so
\[
\HS_{R/I}(t)=\frac{(1-t^2)^n}{(1-t)^n}=(1+t)^n,
\]
whence $\dim_kR/I=2^n$. Localize at $\mfrak$: the ring $A=R_\mfrak$ is regular local of dimension $n$, hence Cohen--Macaulay and formally equidimensional, and $IA$ is generated by a system of parameters, so
\[
e(IA)=\len(A/IA)=\dim_kR/I=2^n .
\]
On the other hand $e(\mfrak^2A)=2^ne(\mfrak A)=2^n$. Since $I\subseteq\mfrak^2$ and both are $\mfrak$-primary, it follows from Theorem \ref{thm:rees} that $\overline{IA}=\overline{\mfrak^2A}=\mfrak^2A$. Integral closure commutes with localization \cite[Proposition 1.1.4]{SH06}, so $\bI A=\mfrak^2A$; contracting to $R$ and using that the homogeneous $\mfrak$-primary ideals $\bI$ and $\mfrak^2$ are contracted from $A$, we get $\bI=\mfrak^2$.

\emph{(2)$\Rightarrow$(1)} is trivial.

It remains to prove (3)$\Leftrightarrow$(4); by the equivalence of (2) and (3) with $V(I)=\{0\}$, established above, it suffices to show
\[
V(I)\neq\{0\}\iff I\ \text{has a witness}.
\]

Suppose $p\in V(I)\setminus\{0\}$ and let $S=\supp(p)\ne\emptyset$. By Lemma \ref{lem:stratify}, $S$ is closed and the system $q^{w_i}=-\epsilon_i$ $(i\in S)$ has the solution $q=p|_S$ in $(k^\times)^S$. By Lemma \ref{lem:divisible} applied with $A=k^\times$, $N=m=|S|$ and $c_i=-\epsilon_i$, this forces $\sigma_S(\lambda)=\prod_{i\in S}(-\epsilon_i)^{\lambda_i}=1$ for all $\lambda\in L_S$. So $S$ is a witness.

Conversely, let $S$ be a witness. Since $\sigma_S\equiv1$ on $L_S$, Lemma \ref{lem:divisible} produces $q\in(k^\times)^S$ with $q^{w_i}=-\epsilon_i$ for all $i\in S$. Extending $q$ by zero to a point $p\in k^n$ with $\supp(p)=S\ne\emptyset$, Lemma \ref{lem:stratify} gives $p\in V(I)\setminus\{0\}$.
\end{proof}

\subsection{First consequences}

\begin{rem}[Parity reduction]\label{rem:parity}
Because $\epsilon_i^2=1$, the value $(-\epsilon_i)^{\lambda_i}$ depends only on the parity of $\lambda_i$: it is $1$ for $\lambda_i$ even and $-\epsilon_i$ for $\lambda_i$ odd. Hence $\sigma_S(\lambda)$ depends only on $\lambda\bmod2$. On the other side, reducing $w_i$ modulo $2$ kills the term $2e_i$ and leaves
\[
\bar w_i=e_{a_i}+e_{b_i}\in\FF_2^S,
\]
the indicator vector of the two-element set $\{a_i,b_i\}$. Let $\overline{M_S}$ be the reduction of $M_S$ modulo $2$. In particular:
\begin{enumerate}[(i)]
\item if $m_i=m_j$ for $i\ne j$ in $S$, then $\bar w_i=\bar w_j$, so $e_i+e_j\in\ker\overline{M_S}$;
\item $\rank_\ZZ M_S\geq\rank_{\FF_2}\overline{M_S}$, and $\lambda\bmod2\in\ker\overline{M_S}$ for every $\lambda\in L_S$.
\end{enumerate}
Notice that $L_S\bmod2$ can be a proper subspace of $\ker\overline{M_S}$ (see Case C in \S\ref{subsec:C}), so (ii) is only a bound.
\end{rem}

The following lemma is the practical tool for evaluating condition (4): it reduces the determination of $L_S$ to exhibiting one relation and a rank count.

\begin{lem}\label{lem:primitive}
Let $S$ be closed, $N=|S|$, and suppose $\lambda_0\in L_S$ is primitive \textup{(}i.e. $\gcd$ of its entries is $1$\textup{)} and $\rank_\ZZ M_S=N-1$. Then $L_S=\ZZ\lambda_0$, and $S$ is sign-trivial if and only if $\sigma_S(\lambda_0)=1$. If moreover $\dim_{\FF_2}\ker\overline{M_S}=1$, then $\lambda_0\bmod2$ is the unique nonzero element of $\ker\overline{M_S}$.
\end{lem}

\begin{proof}
$L_S=\ker M_S$ is a saturated subgroup of $\ZZ^S$ of rank $N-\rank M_S=1$, hence $L_S=\ZZ\mu$ for a primitive $\mu$. Writing $\lambda_0=c\mu$ and using primitivity of $\lambda_0$ gives $c=\pm1$, so $L_S=\ZZ\lambda_0$. As $\sigma_S$ is a homomorphism, $\sigma_S\equiv1$ iff $\sigma_S(\lambda_0)=1$. Finally $\lambda_0\bmod2\ne0$ by primitivity and $\lambda_0\bmod2\in\ker\overline{M_S}$ by Remark \ref{rem:parity}(ii).
\end{proof}

Three consequences of Theorem~\ref{thm:main-intro} for arbitrary $n$ are immediate. The first confirms the experimental observation about negative signs.

\begin{cor}\label{cor:allminus}
If $\epsilon_i=-1$ for all $i$, then $I$ is not weakly NND. More generally, in characteristic $2$ no ideal as in Setup \ref{setup} is weakly NND.
\end{cor}

\begin{proof}
In both cases $-\epsilon_i=1$ for all $i$, so $\sigma_S\equiv1$ for every $S$; the full set $S=\{1,\dots,n\}$ is closed, since its first condition is automatic and its second is vacuous, hence it is a witness. Concretely, $p=(1,1,\dots,1)$ satisfies $g_i(p)=1+\epsilon_i=0$, so $V(I)\ne\{0\}$.
\end{proof}

\begin{cor}\label{cor:constant}
Suppose all generators involve the same monomial, $m_1=\dots=m_n=x_ax_b$ with $a\ne b$, and $\operatorname{char}k\ne2$. Then
\[
I\ \text{is weakly NND}\iff\epsilon_a\epsilon_b=-1 .
\]
In particular, for every $n\geq2$ the ideal $I=(x_1^2-x_1x_2,\ x_2^2+x_1x_2,\ x_3^2-x_1x_2,\ \dots,\ x_n^2-x_1x_2)$ is weakly NND, so the family of Setup \ref{setup} contains weakly NND members for all $n$.
\end{cor}

\begin{proof}
If $S$ is nonempty and closed then $a,b\in S$ (take any $i\in S$), so no index can lie outside $S$ (the second closure condition would fail for it); hence $S=\{1,\dots,n\}$ is the only candidate, and it is closed. Here $w_a=e_a-e_b$, $w_b=e_b-e_a$, and $w_i=2e_i-e_a-e_b$ for $i\neq a,b$. In a relation $\sum_i\lambda_iw_i=0$, the $i$-th coordinate for $i\ne a,b$ is $2\lambda_i$, forcing $\lambda_i=0$; the remaining condition $(\lambda_a-\lambda_b)(e_a-e_b)=0$ gives $\lambda_a=\lambda_b$. Thus $\rank M_S=n-1$ and $\lambda_0=e_a+e_b$ is a primitive relation, so $L_S=\ZZ(e_a+e_b)$ by Lemma \ref{lem:primitive}. Now $\sigma_S(\lambda_0)=(-\epsilon_a)(-\epsilon_b)=\epsilon_a\epsilon_b$, and $S$ fails to be a witness precisely when $\epsilon_a\epsilon_b\ne1$, i.e. $\epsilon_a\epsilon_b=-1$ since $\operatorname{char}k\ne2$.
\end{proof}

\begin{cor}\label{cor:cycle}
Let $\operatorname{char}k\ne2$ and let $I$ be the \emph{cycle ideal}
\[
I=(g_1,\dots,g_n),\qquad g_i=x_i^2+\epsilon_ix_ix_{i+1}\quad(\text{indices modulo }n).
\]
Then $I$ is weakly NND if and only if $\prod_{i=1}^n\epsilon_i=(-1)^{n+1}$.
\end{cor}

\begin{proof}
Here $\{a_i,b_i\}=\{i,i+1\}$, so a nonempty closed $S$ satisfies $i\in S\Rightarrow i+1\in S$ and hence $S=\{1,\dots,n\}$, which is closed. From $w_i=e_i-e_{i+1}$ we get, comparing the coefficient of $e_j$ in $\sum_i\lambda_iw_i=0$, that $\lambda_j=\lambda_{j-1}$ for all $j$; thus $\rank M_S=n-1$ and $L_S=\ZZ(1,\dots,1)$ by Lemma \ref{lem:primitive}. Now $\sigma_S(1,\dots,1)=\prod_i(-\epsilon_i)=(-1)^n\prod_i\epsilon_i$, and $S$ fails to be a witness exactly when this is $\ne1$, i.e. $\prod_i\epsilon_i=(-1)^{n+1}$.
\end{proof}

\section{The case $n=3$}\label{sec:n3}

Throughout this section $n=3$, $\operatorname{char}k\ne2$, and we write
\[
I=(x^2+\epsilon_1m_1,\ y^2+\epsilon_2m_2,\ z^2+\epsilon_3m_3),\qquad m_i\in\{xy,xz,yz\},
\]
identifying positions $1,2,3$ with the variables $x,y,z$. There are $3^3\cdot2^3=216$ such ideals. We now use condition (4) of Theorem \ref{thm:main-intro} to characterize completely which of these ideals are weakly NND; the support patterns fall into three cases, and each is settled by Lemma \ref{lem:primitive}.

\begin{lem}\label{lem:closedn3}
Let $S\subseteq\{1,2,3\}$ be nonempty and closed. Then:
\begin{enumerate}[(i)]
\item $|S|\neq1$;
\item $S=\{1,2,3\}$ is always closed;
\item $S=\{i,j\}$ with $i\ne j$ is closed if and only if $m_i=m_j=x_ix_j$ and $m_l\ne x_ix_j$, where $l$ is the third index.
\end{enumerate}
\end{lem}

\begin{proof}
(i) If $S=\{i\}$, closedness requires $a_i,b_i\in S$, impossible since $a_i\ne b_i$.

(ii) For $S=\{1,2,3\}$ the first condition holds trivially and the second is vacuous.

(iii) For $S=\{i,j\}$ and $k\in S$, closedness requires $\{a_k,b_k\}\subseteq\{i,j\}$; as $a_k\ne b_k$ this means $\{a_k,b_k\}=\{i,j\}$, i.e. $m_k=x_ix_j$, for both $k=i$ and $k=j$. The second condition applied to $l\notin S$ says $\{a_l,b_l\}\not\subseteq\{i,j\}$, i.e. $m_l\ne x_ix_j$. Conversely, these conditions clearly give closedness.
\end{proof}

There is a subtlety in (iii): a pair $S=\{i,j\}$ is closed not merely when the two positions $i,j$ carry the same monomial, but when that shared monomial is exactly $x_ix_j$, so that its variable support coincides with the set of positions carrying it. This coincidence splits the \enquote{one repeated monomial} case into two genuinely different sub-cases below. Since $\{m_1,m_2,m_3\}$ either consists of three distinct monomials, or has exactly one repeat, or is constant, we obtain three cases.

\subsection{Case A: $m_1,m_2,m_3$ pairwise distinct}\label{subsec:A}

Here $\{m_1,m_2,m_3\}=\{xy,xz,yz\}$, so $i\mapsto\{a_i,b_i\}$ is a bijection $\tau$ from positions to the three $2$-subsets of $\{1,2,3\}$. By Lemma \ref{lem:closedn3}(iii) no pair is closed (a pair would need a repeated monomial), so $S=\{1,2,3\}$ is the only closed set.

Since $\tau$ is a bijection, each index $l$ lies in exactly two of the three sets $\tau(1),\tau(2),\tau(3)$, whence
\[
\sum_{i=1}^3w_i=2(e_1+e_2+e_3)-\sum_{i=1}^3(e_{a_i}+e_{b_i})=2(e_1+e_2+e_3)-2(e_1+e_2+e_3)=0,
\]
\emph{for every bijection $\tau$}. So $\lambda_0=(1,1,1)$ is a primitive element of $L_S$. Moreover $\bar w_1,\bar w_2,\bar w_3$ are the three distinct nonzero indicator vectors $(1,1,0),(1,0,1),(0,1,1)$ of $\FF_2^3$, which span a $2$-dimensional space; hence $\rank_\ZZ M_S\geq2$ by Remark \ref{rem:parity}(ii), and $\rank_\ZZ M_S=2$ because $L_S\ne0$. Lemma \ref{lem:primitive} gives $L_S=\ZZ(1,1,1)$ and
\[
\sigma_S(1,1,1)=(-\epsilon_1)(-\epsilon_2)(-\epsilon_3)=-\epsilon_1\epsilon_2\epsilon_3 .
\]
So $S$ is a witness iff $-\epsilon_1\epsilon_2\epsilon_3=1$, and Theorem~\ref{thm:main-intro} yields
\[
\boxed{\ I\ \text{is weakly NND}\iff\epsilon_1\epsilon_2\epsilon_3=+1.\ }
\]
This is independent of which of the six bijections $\tau$ occurs.

\subsection{Case B: exactly one monomial is repeated, at two positions}\label{subsec:B}

Say $m_i=m_j$ for two positions $i\ne j$, and $m_l\ne m_i$ for the third position $l$. Whether the pair $\{i,j\}$ is closed depends on whether the shared monomial is $x_ix_j$.

\subsubsection*{Case B1: the shared monomial is $m_i=m_j=x_ix_j$}
Now $S=\{i,j\}$ is closed by Lemma \ref{lem:closedn3}(iii), since $m_l\ne m_i=x_ix_j$. In the coordinates $(\lambda_i,\lambda_j)$ on $\ZZ^S$ we have $w_i=2e_i-e_i-e_j=e_i-e_j$ and $w_j=e_j-e_i=-w_i$, so $\rank M_S=1$ and $\lambda_0=(1,1)$ is a primitive relation; by Lemma \ref{lem:primitive}, $L_S=\ZZ(1,1)$ and
\[
\sigma_S(1,1)=(-\epsilon_i)(-\epsilon_j)=\epsilon_i\epsilon_j .
\]
We must also test the always-closed set $S'=\{1,2,3\}$. Since $w_l$ has nonzero $l$-th coordinate while
$w_i$ and $w_j$ do not, a relation $\lambda_iw_i+\lambda_jw_j+\lambda_lw_l=0$ forces $\lambda_l=0$ and,
as before, $\lambda_i=\lambda_j$; thus $L_{S'}=\ZZ(1,1,0)$ is simply the pair relation padded by a zero,
and $\sigma_{S'}(1,1,0)=\epsilon_i\epsilon_j$ again. Both closed sets give the same condition, so
\[
\boxed{\ I\ \text{is weakly NND}\iff\epsilon_i\epsilon_j=-1,\qquad\epsilon_l\ \text{free}.\ }
\]

\subsubsection*{Case B2: the shared monomial is not $x_ix_j$}
Now $S=\{i,j\}$ fails closedness, and no other pair is closed either: a closed pair $\{i',j'\}$ needs $m_{i'}=m_{j'}$, and $\{i,j\}$ is the only pair of positions carrying equal monomials. So $S=\{1,2,3\}$ is the only closed set.

Since $m_i=m_j$ and $m_l\ne m_i$, the reductions $\bar w_i=\bar w_j$ and $\bar w_l$ are distinct nonzero indicator vectors in $\FF_2^3$, hence independent: $\rank_{\FF_2}\overline{M_S}=2$ and $\dim_{\FF_2}\ker\overline{M_S}=1$, with
\[
\ker\overline{M_S}=\{0,\,e_i+e_j\} .
\]
By Remark \ref{rem:parity}(ii), $\rank_\ZZ M_S\geq2$. To see that the rank is exactly $2$ we exhibit a relation. Since $m_i=m_j\ne x_ix_j$, after possibly exchanging the names of $i$ and $j$ we may assume $m_i=m_j=x_ix_l$. Then $w_i=e_i-e_l$ and $w_j=2e_j-e_i-e_l$, so $w_i+w_j=2(e_j-e_l)$. There are two possibilities for $m_l$:
\begin{itemize}
\item $m_l=x_jx_l$: then $w_l=e_l-e_j$ and $w_i+w_j+2w_l=0$, giving $\lambda_0=(1,1,2)$;
\item $m_l=x_ix_j$: then $w_l=2e_l-e_i-e_j$ and $w_i+w_j+2w_l=2e_l-2e_i=-2w_i$, giving $3w_i+w_j+2w_l=0$, i.e. $\lambda_0=(3,1,2)$;
\end{itemize}
in the coordinates $(\lambda_i,\lambda_j,\lambda_l)$. In both cases $\lambda_0$ is primitive and $\rank_\ZZ M_S=2$, so $L_S=\ZZ\lambda_0$ by Lemma \ref{lem:primitive}. Since $\lambda_0\equiv e_i+e_j\pmod2$, the parity reduction of Remark \ref{rem:parity} gives $\sigma_S(\lambda_0)=\epsilon_i\epsilon_j$, and hence, exactly as in Case B1,
\[
\boxed{\ I\ \text{is weakly NND}\iff\epsilon_i\epsilon_j=-1,\qquad\epsilon_l\ \text{free}.\ }
\]
\subsection{Case C: $m_1=m_2=m_3=x_ix_j$}\label{subsec:C}

Let $l$ be the third index. No pair is closed: for the pair $\{i,j\}$ the condition $m_l\ne x_ix_j$ fails, and for a pair containing $l$, say $\{i,l\}$, we would need $m_i=x_ix_l\ne x_ix_j$. So $S=\{1,2,3\}$ is the only closed set. Here the bound modulo $2$ is not sharp: the three vectors $\bar w_1=\bar w_2=\bar w_3=e_i+e_j$ coincide, so $\dim_{\FF_2}\ker\overline{M_S}=2$ and the last clause of Lemma \ref{lem:primitive} does not apply. We therefore argue over $\ZZ$ directly. Here
\[
w_i=e_i-e_j,\qquad w_j=e_j-e_i=-w_i,\qquad w_l=2e_l-e_i-e_j,
\]
and $w_i,w_l$ are independent (the $l$-th coordinate of $w_l$ is $2$, that of $w_i$ is $0$), so $\rank_\ZZ M_S=2$ and, as in Case B1, $L_S=\ZZ(1,1,0)$ in the coordinates $(\lambda_i,\lambda_j,\lambda_l)$. Consequently $\sigma_S(1,1,0)=\epsilon_i\epsilon_j$ and
\[
\boxed{\ I\ \text{is weakly NND}\iff\epsilon_i\epsilon_j=-1,\qquad\epsilon_l\ \text{free},\ }
\]
where now $\{i,j\}$ is the pair of positions whose own indices occur in the common monomial $x_ix_j$, and $l$ is the third, unrelated position.

\subsection{The classification}

\begin{cor}\label{cor:n3}
Let $\operatorname{char}k\ne2$ and $I=(x^2+\epsilon_1m_1,\,y^2+\epsilon_2m_2,\,z^2+\epsilon_3m_3)$, where each $m_i$ belongs to $\{xy,xz,yz\}$.
\begin{enumerate}[(i)]
\item If $m_1,m_2,m_3$ are pairwise distinct, then $I$ is weakly NND if and only if $\epsilon_1\epsilon_2\epsilon_3=+1$.
\item If some monomial repeats, let the \emph{critical pair} $\{i,j\}$ be
\begin{itemize}
\item the two positions carrying the repeated monomial, if a monomial repeats exactly twice;
\item the two positions $i,j$ with $m_1=m_2=m_3=x_ix_j$, if a monomial repeats three times.
\end{itemize}
Then $I$ is weakly NND if and only if $\epsilon_i\epsilon_j=-1$; the third sign is free.
\end{enumerate}
\end{cor}

\begin{proof}
Cases A, B1, B2, C of \S\S\ref{subsec:A}--\ref{subsec:C} are exhaustive by Lemma \ref{lem:closedn3} and the trichotomy on $(m_1,m_2,m_3)$, and prove (i) and (ii).
\end{proof}

\section{Examples at $n=4$}\label{sec:n4}

At $n=4$ the interplay between the closed sets and the sign character is richer than for $n=3$. In the first example below, the full set is the only closed set, so it alone has to be tested. In the second, a proper subset is a witness, and it is strictly larger than the pair of positions carrying the repeated monomial, so it cannot be read off from the support pattern by inspection. Remark \ref{rem:quantifier} then shows that the quantifier over $S$ in Theorem~\ref{thm:main-intro}(4) is genuinely necessary: a proper subset can be a witness while the full set is not. All ideals below lie in $R=k[x_1,x_2,x_3,x_4]$.

\begin{exmp}[Degenerate, full support $S=\{1,2,3,4\}$]\label{ex:full4}
\[
I=(x_1^2+x_2x_3,\ x_2^2+x_3x_4,\ x_3^2+x_1x_4,\ x_1x_2+x_4^2).
\]
No proper nonempty subset of $\{1,2,3,4\}$ satisfies the closure condition of Theorem~\ref{thm:main-intro}(4), so $S=\{1,2,3,4\}$ is the only closed set. Here $\sum_{i=1}^4w_i=0$, so $L_S=\ZZ(1,1,1,1)$ and $\sigma_S(1,1,1,1)=1$: the set $S$ is sign-trivial. Accordingly $p=t\cdot(-1,1,-1,1)$ lies in $V(I)$ for all $t\in k$, so $\dim V(I)\geq1$. Hence $\bI\ne\mfrak^2$, and therefore $I$ is not weakly NND.
\end{exmp}

The next example extends the phenomena of Section \ref{sec:n3} in a direction that cannot occur for $n=3$: a witness strictly larger than the pair of positions carrying the repeated monomial.

\begin{exmp}[A proper witness not determined by the repeated-monomial positions]\label{ex:nonadjacent4}
\[
I=(x_1^2+x_1x_4,\ x_2^2+x_1x_4,\ x_3^2+x_1x_3,\ x_4^2+x_1x_2).
\]
Here $m_1=m_2=x_1x_4$ while $m_4=x_1x_2$, so the positions carrying the repeated monomial are $\{1,2\}$, as in Case B2 of \S\ref{subsec:B}. This pair is not closed, since $m_1=x_1x_4$ forces $4\in S$; the smallest closed set containing it is $S=\{1,2,4\}$, strictly larger than $\{1,2\}$ and hence not visible on inspection of which generators share a monomial. In the coordinates $(\lambda_1,\lambda_2,\lambda_4)$ on $S$ the relation lattice is generated by $(3,1,2)$, exactly the exponent pattern of Case B2, and $\sigma_S$ evaluates to $\epsilon_1\epsilon_2$ with $\epsilon_4$ free. Taking $\epsilon_1=\epsilon_2=\epsilon_4=+1$ makes $S$ sign-trivial, and indeed
\[
L=(x_3,\ x_1+x_2,\ x_2-x_4)
\]
lies in $V(I)$ (explicitly, $(x_1,x_2,x_3,x_4)=(-t,t,0,t)$ for $t\in k$): one checks directly that all four generators vanish along $L$. Here $\dim V(I)=1$ and $\bI$ is not a monomial ideal. In this example the full set happens to be a witness as well, since $L_{\{1,2,3,4\}}=\ZZ(3,1,0,2)$ and $\sigma(3,1,0,2)=\epsilon_1\epsilon_2=1$; the next remark shows that this need not happen.
\end{exmp}

\begin{rem}\label{rem:quantifier}
The quantifier over $S$ in Theorem~\ref{thm:main-intro}(4) cannot be replaced by the single test $S=\{1,\dots,n\}$. For $I=(x_1^2+x_1x_2,\,x_2^2+x_1x_2,\,x_3^2+x_3x_4,\,x_4^2-x_3x_4)$ the pair $S=\{1,2\}$ is closed with $L_S=\ZZ(1,1)$ and $\sigma_S(1,1)=\epsilon_1\epsilon_2=1$, so $S$ is a witness and $I$ is not weakly NND; accordingly $(-t,t,0,0)\in V(I)$ for all $t\in k$. The full set is closed as well, but $L_{\{1,2,3,4\}}=\ZZ(1,1,0,0)+\ZZ(0,0,1,1)$ and $\sigma(0,0,1,1)=\epsilon_3\epsilon_4=-1$, so it is \emph{not} sign-trivial: testing it alone would have missed the degeneracy.
\end{rem}

Condition (4) is decidable: there are at most $2^n-1$ subsets $S$ to inspect, and for each one, computing $L_S$ is a matter of integer linear algebra. In practice far fewer than $2^n-1$ sets need to be tested, since the closure condition is quite restrictive. This restrictiveness has a clean combinatorial description. Encode the support data as a directed hypergraph on $\{1,\dots,n\}$, with an edge $i\to\{a_i,b_i\}$ for each $i$. Then the closed sets $S$ are exactly the subsets that are closed under following edges forward and are never entered from outside. In this language, sign-triviality becomes a character condition on the cycle space of the hypergraph.

Finally, the proof of Theorem~\ref{thm:main-intro} uses essentially that all generators are quadrics with the same Newton polyhedron; in particular, if $\bI$ is a monomial ideal then it must be $\mfrak^2$. The natural next family is $I=(x_1^{d}+\epsilon_1m_1,\dots,x_n^{d}+\epsilon_nm_n)$ with $\deg m_i=d$ and $m_i$ not a power of a single variable.

\begin{quest}
Is there an algorithm deciding condition (4) in time polynomial in $n$? And does the equivalence \enquote{weakly NND $\iff$ regular sequence} persist in degree $d$, with condition (4) suitably modified?
\end{quest}

\medskip

\noindent\textbf{AI usage statement.} Claude Opus 5 (Anthropic) was used in preparing this manuscript in two capacities: to computationally verify the classification of Corollary \ref{cor:n3}, by exhaustively checking condition (4) of Theorem \ref{thm:main-intro} against the general criterion for all $216$ ideals of Setup \ref{setup} with $n=3$, and to check the closed sets, relation lattices, and witnesses for the example ideals with $n=4$ in Section \ref{sec:n4}; and to improve the readability of the exposition and correct grammatical errors. All mathematical statements, proofs, and results are the authors' own, who have reviewed and verified all AI-assisted contributions and take full responsibility for the content of this paper.

\medskip

\noindent\textbf{Data availability statement:} All data generated or analysed during this study are included in this published article (and its supplementary information files). 

\medskip

\noindent\textbf{Conflict of interest statement:} All authors declare that they have no conflicts of interest.


\end{document}